\documentclass[12pt,reqno]{amsart}
\usepackage[T1]{fontenc}
\usepackage{graphicx}
\usepackage{todonotes}
\usepackage{etoolbox}

\usepackage{color}
\definecolor{MyLinkColor}{rgb}{0,0,0.4}
\usepackage{hyperref}

\definecolor{MyLinkColor}{rgb}{0,0,0.4}
\newcommand{\tbe}{\color{black} } 

\newcommand{\R}{{\mathbb R}}

\newcommand{\C}{{\mathcal C}}

\newcommand{\la}{\lambda}

\newcommand{\vp}{\phi}
\newcommand{\ve}{\varepsilon}
\newcommand{\bu}{\bar u}
\newtheorem{thm}{Theorem}

\theoremstyle{definition}
\newtheorem{Definition}[thm]{Definition}
\theoremstyle{remark} 
\newtheorem{rem}[thm]{Remark}

\makeatletter
\patchcmd{\maketitle}{\@fnsymbol}{\@alph}{}{}  % Footnote numbers from symbols to small letters
\makeatother

\title[Symmetric waves are traveling waves]{Symmetric waves are traveling waves for a shallow water equation for surface waves of moderate amplitude }

\author[A. Geyer]{Anna Geyer}
\address{Faculty of Mathematics, University of Vienna, Oskar-Morgenstern-Platz 1, \mbox{1090 Vienna}, Austria.}
\email{anna.geyer@univie.ac.at}
\date{\today}


\begin{document}

\maketitle

\begin{abstract}
Following a general principle introduced by Ehrnstr\"{o}m et.al.~in \cite{Ehrnstrom2009a}, we prove that for an equation modeling the free surface evolution of moderate amplitude waves in shallow water, all symmetric waves are traveling waves. 
\end{abstract}

\section{Introduction}
\label{s_intro}

In this note we are concerned with the relation between symmetric and traveling wave solutions of  a model equation for surface waves of moderate amplitude in shallow water 
\begin{equation}
\label{eq_MASE}
 u_t+u_x+6uu_x-6u^2u_x+12 u^3u_x+u_{xxx}-u_{xxt} +14 u u_{xxx}+28u_xu_{xx}=0,
\end{equation}
 which arises as an approximation of the Euler equations in the context of homogenous,
 inviscid gravity water waves propagating over a flat bed. The equation was originally derived by Johnson \cite{Joh02} by means of formal asymptotic expansions from the Euler equations. These considerations were followed up by Constantin and Lannes \cite{ConLan09} who proved that the equation approximates the governing equations to the same order as the Camassa-Holm equation \cite{Camassa1993}, which  models the horizontal fluid velocity at a certain depth beneath the fluid  \cite{Joh02}.  An  important feature of equation \eqref{eq_MASE} is that it captures the non-linear phenomenon of wave breaking, in the sense that the surface profile remains bounded but its slope may form singularities in finite time \cite{ConLan09,DurukMutlubas2013a}. The Cauchy problem associated to \eqref{eq_MASE} is locally well-posed in $H^s$, for $s>3/2$,  on the real line as well as  on the circle \cite{DGM14, DurukMutlubas2013b}, and the data-to-solution map was shown  to be non-uniformly continuous \cite{DGM14}. For a well-posedness result in the context of Besov spaces we refer the reader to \cite{Mi2013}, whereas a result on low regularity solutions with  $1< s \leq 3/2$ may be found in  \cite{Liu2014a}, and for global conservative solutions and continuation of solutions beyond wave breaking we point out \cite{Zhou2014}.
 The equation admits smooth as well as cusped and peaked solitary and periodic traveling wave solutions, and also solitary traveling waves with compact support \cite{Gasull2014,Gey12c,GeyMan15}. Furthermore, the smooth solitary traveling wave solutions  are known to be orbitally stable \cite{DurukGeyer2013a}.\\

In the present note we are concerned with symmetry properties of solutions of \eqref{eq_MASE}. 
A vast amount of research has been conducted into considerations regarding symmetry in the context of water waves, for model equations as well for the full governing equations. 
It is astonishing that all steady water waves known to exist are symmetric and that in many cases symmetry can be established a priori, see for example \cite{Craig1988, Garabedian1965, Kogelbauer2015} for solitary and periodic waves  in irrotational flows, 
 \cite{ConEhrWah07,Constantin2004,ConStr04}  for periodic water waves with vorticity,     \cite{walsh2009} for stratified water waves, 
 and \cite{Henry2014} for  equatorial wind waves. A number of characterizations of symmetric waves    in various settings also point towards the importance of this defining feature in the analysis of water waves, see for instance \cite{Matioc2012} for  rotational solitary waves,  \cite{Tulzer2012} for periodic water waves with stagnation points, and \cite{Matioc2014a}  for a  characterization of symmetric steady water waves in terms of the underlying flow. \\

Following a general principle introduced by Ehrnstr\"{o}m et.al.~in \cite{Ehrnstrom2009a}, our main result states that a symmetric wave solution of equation \eqref{eq_MASE} is necessarily a traveling wave. Conversely, we know from the analysis carried out in \cite{Gasull2014,GeyMan15} that all solitary and periodic traveling wave solutions (smooth as well as peaked or cusped) are symmetric about their crest and trough. In light of these considerations we establish the identity between symmetric and  solitary and periodic traveling wave solutions of \eqref{eq_MASE}, cf.~Remark \ref{rem}. We point out that similar results hold for the Camassa-Holm equation \cite{Ehrnstrom2009a} and for the Whitham equation \cite{BruEhrPei15}. 

\section{Main result}
\label{s_mainresult}

Following \cite{Ehrnstrom2009a} and the general principle deduced therein, we make the following definition for  symmetric waves.

\begin{Definition}
\label{def_symmsol}
A solution $u$ is \emph{$x$-symmetric} if there exists a function $\lambda \in \C^1(\R_+)$ such that for every $t>0$, 
\begin{equation*}
    u(t,x)=u(t,2\lambda(t)-x)
\end{equation*}
for a.e.~$x\in\R$. We say that $\lambda(t)$ is the \emph{axis of symmetry}. 
\end{Definition}

The equation for surface waves \eqref{eq_MASE} fulfills the formal requirements of the general principle stated in Theorem 2.2 of \cite{Ehrnstrom2009a} whose proof is carried out for classical solutions. It is possible, however, to modify the argument to extend the solution space beyond differentiable functions and to accommodate for weak solutions. To this end, we  write the equation \eqref{eq_MASE} in weak non-local form as
 \begin{equation}
  \label{eq_MASEnonloc}
  u_t  - \partial_x(u +7u^2) + \partial_x (1-\partial_x^2)^{-1} R(u)= 0, \quad x\in \R,\; t> 0,
  \end{equation}
where
\begin{equation}
    R(u):=2u+10u^2 -2u^3+3u^4-7u_x^2.
\end{equation}
Note that this equation holds in $H^1(\R)$. Indeed, let $P(x):= (1-\partial_x^2)^{-1} R(u) = \frac{1}{2} e^{-|x|}* R(u)$, then $P\in H^1(\R)$ for $u\in H^1(\R)$ by Young's inequality. 
We define weak solutions of equation \eqref{eq_MASE} in the  following way.

\begin{Definition}
\label{def_weak}
A function $u(t,x)$ is a weak solution of equation \eqref{eq_MASE} if $u\in\C(\R_+,H^1(\R))$ satisfies
%\begin{equation*}
%\iint_{\R_+ \times \R}  u \vp_t -(u+7u^2) \vp_x + \partial_x (1-\partial_x^2)^{-1}  \Big(2u +10u^2-2u^3+3u^4-7u_x^2\Big)\vp_x dt dx=0, 
%\end{equation*}
\begin{equation}
\label{eq_MASEweak}
\iint_{\R_+ \times \R}  u \vp_t -(u+7u^2) \vp_x + (1-\partial_x^2)^{-1}R(u)\vp_x dt dx=0, 
\end{equation}
for all $\vp\in \C^{\infty}_0(\R_+ \times \R)$.
\end{Definition}

We are now ready to state our main result. 

\begin{thm}
\label{MT}
Let $u$ be a weak solution of \eqref{eq_MASE} with initial data such that the equation is locally well-posed. If $u$ is $x$-symmetric then $u$ is a traveling wave solution. 
\end{thm}

\begin{proof}
We follow  \cite{Ehrnstrom2009a} and introduce the notation 
\begin{equation}
    f_{\la} (t,x) := f(t,2\la(t) -x).
\end{equation}
Let $\vp\in   \C_0^1(\R_+,\C_0^3)$ be a test function and
note that the Definition \ref{def_weak} for weak solutions of \eqref{eq_MASE} remains valid, since  $\C^{\infty}_0(\R_+ \times \R)$ is dense in  $ \C_0^1(\R_+,\C_0^3)$. Furthermore,  $\C_0^1(\R_+,\C_0^3)$ is invariant  under $\vp \rightarrow \vp_{\la}$ because $\la\in\C^1$, and in view of $(\vp_{\la})_{\la} = \vp$ we find that this transformation is bijective. Let $u$ be an $x$-symmetric solution of \eqref{eq_MASE}, then,  using bracket notation for distributions, $u$  satisfies 
  \begin{equation}
    \label{eq_MASEweakDistr}
      <u,\vp_t> - <u+7u^2,\vp_x>+<(1-\partial_x^2)^{-1}R(u),\vp_x> = 0.   
  \end{equation}
Since $u$ is symmetric, we have $u=u_{\la}$ and we find that   $<u_{\la},\vp> = -<u,\vp_{\la}>$, $<u_{\la}^2,\vp> =  -<u^2,\vp_{\la}>$ and 
\begin{align*}
    <(1-\partial_x^2)^{-1}R(u_{\la}),\vp> 
        &=\iint_{\R_+ \times \R} \frac{1}{2} e^{-|x|}\ast R(u(t,2\la-x))  \vp(t,x)dt dx \\
        &=\iint_{\R_+ \times \R} \int_{\R}\frac{1}{2} e^{-|s|}R(u(t,2\la -(x-s)))ds\, \vp(t,x)dt dx \\
        &=-\iint_{\R_+ \times \R} \int_{\R}\frac{1}{2} e^{-|s|}R(u(t,x+s))ds\, \vp(t,2\la-x)dt dx \\
        &=-\iint_{\R_+ \times \R} \frac{1}{2} e^{-|x|}\ast R(u(t,x))  \vp(t,2\la-x)dt dx \\
    &=  -<(1-\partial_x^2)^{-1}R(u),\vp_{\la}>.
\end{align*}
Furthermore, we have that
\begin{align*}
      (\vp_{\la})_t= (\vp_t)_{\la}-2\dot\la(t) (\vp_x)_{\la}, \quad (\vp_{\la})_x=-(\vp_x)_{\la}.
\end{align*}
Since $u_{\la}$ satisfies \eqref{eq_MASEweakDistr} we have in view of the above identities that
\begin{align*}
      0&=<u,(\vp_t)_{\la}> - <u+7u^2,(\vp_x)_{\la}>+<(1-\partial_x^2)^{-1}R(u),(\vp_x)_{\la}> \\
      &=<u,(\vp_{\la})_t-2\dot\la(t)(\vp_{\la})_x> + <u+7u^2,(\vp_{\la})_x>-<(1-\partial_x^2)^{-1}R(u),(\vp_{\la})_x>.
\end{align*}
Taking $\vp=\vp_{\la}$ in the above equation, we get
\begin{equation*}
   <u,\vp_t-2\dot\la(t)\vp_x> + <u+7u^2,\vp_x>-<(1-\partial_x^2)^{-1}R(u),\vp_x>=0,
\end{equation*}
 since $(\vp_{\la})_{\la} = \vp$. We subtract the latter equation from \eqref{eq_MASEweakDistr} and obtain
\begin{equation}
\label{eq_eq}
   <u,\dot\la(t)\vp_x> - <u+7u^2,\vp_x>+<(1-\partial_x^2)^{-1}R(u),\vp_x>=0.
\end{equation}
Consider now  the sequence 
\begin{equation*}
    \vp_{\ve} (t,x)=\psi(x) \rho_{\ve}(t)
\end{equation*}
where $\rho_{\ve}\in \C^{\infty}_0(\R)$ is a mollifier with the property that  $\rho_{\ve}(t) \rightarrow \delta(t-t_0)$, the Dirac delta function with mass in $t_0$, as $\ve \rightarrow 0$. Using this sequence in equation \eqref{eq_eq} we obtain
\begin{align}
\label{eq_seqeq}
    \int_{\R}\psi(x) \Big[\int_{\R_+}\dot\la(t) u(t,x)\rho_{\ve}(t) -(u+7u^2)(t,x)\rho_{\ve}(t)  + (1-\partial_x^2)^{-1}R(u(t,x))\rho_{\ve}(t)dt\Big] dx=0.
\end{align}
By assumption we have that $u\in\C(\R_+,H^1(\R))$, and therefore
\begin{equation*}
    \lim_{\ve \rightarrow 0}\int_{\R_+}u(t,x)\rho_{\ve}(t) dt = u(t_0,x)
\end{equation*}
in $L^2(\R)$ and 
\begin{equation*}
    \lim_{\ve \rightarrow 0}\int_{\R_+}u^2(t,x)\rho_{\ve}(t) dt = u^2(t_0,x)
\end{equation*}
in $L^1(\R)$. Furthermore, we have that 
\begin{equation*}
    \lim_{\ve \rightarrow 0}\int_{\R_+} (1-\partial_x^2)^{-1}R(u(t,x))\rho_{\ve}(t) dt =(1-\partial_x^2)^{-1}R(u (t_0,x))
\end{equation*}
in $L^2(\R)$, since $(1-\partial_x^2)^{-1} R(u)\in H^1(\R)\subset L^2(\R)$, cf.~\cite{Rudin}.
Therefore, in the limit  when $\ve$ tends to zero, \eqref{eq_seqeq} implies that $u(t_0,x)$ satisfies \begin{equation}
\label{eq_MASETW}
\int_{\R}  (-\dot\la(t_0)u+u +7u^2) \psi_x - (1-\partial_x^2)^{-1}R(u)\psi_x dx=0, 
\end{equation}
for all $\psi\in \C^{\infty}_0(\R)$.
Set now $c:=-\dot \la(t_0)$ and define  $\bar u(t,x)=u(t_0,x+c(t-t_0))$. We claim that the function $\bar{u}$ is a weak traveling wave solution of \eqref{eq_MASE}. If this is true the result follows immediately, since then $\bar u(t_0,x)=u(t_0,x)$ and in view of uniqueness of solutions we find that $\bar u(t,x)=u(t,x)$ for all $t$. Therefore, $u$ is a traveling wave solution of \eqref{eq_MASE} as we wanted to show. 
To prove the claim, we demonstrate that if a function $U\in H^1(\R)$ satisfies
\begin{equation}
    \label{eq_MASETW}
        \int_{\R}  (cU+U +7U^2) \psi_x - (1-\partial_x^2)^{-1}R(U)\psi_x dx=0, 
\end{equation}
for all $\psi\in \C^{\infty}_0(\R)$, then the function 
\begin{equation*}
     \label{eq_TWsol}
        \bu(t,x)=U(x-c(t-t_0))
\end{equation*}
is a weak solution of \eqref{eq_MASE}. To see that this is true, 
notice that $\bu$ defined as above belongs to $\C(\R, H^1(\R))$ since the translation map is  continuous from $\R$ into $ H^1(\R)$ and $t\rightarrow c(t-t_0)$ is real analytic. We may assume  for convenience that  $t_0=0$ and for any function  $\vp\in \C^{\infty}_0(\R_+\times \R)$ use the shorthand
\begin{equation*}
    \vp_c(t,x):=\vp(t,x+ct).
\end{equation*}
Then we have that $<\bu,\vp>=<U,\vp_c>$ and 
\begin{align*}
    <(1-\partial_x^2)^{-1}R(\bu),\vp>
        &=\iint_{\R_+ \times \R} \frac{1}{2} e^{-|x|}\ast R(U(x-ct))  \vp(t,x)dt dx \\
        &=\iint_{\R_+ \times \R} \int_{\R}\frac{1}{2} e^{-|s|}R(U(x-ct -s))ds\, \vp(t, x)dt dx \\
        &=\iint_{\R_+ \times \R} \int_{\R}\frac{1}{2} e^{-|s|}R(U(r-s))ds\, \vp(t,r+ct)dt dr \\
        &=\iint_{\R_+ \times \R} \frac{1}{2} e^{-|r|}\ast R(U(r))  \vp(t,r+ct)dt dr \\
        &= <(1-\partial_x^2)^{-1}R(U),\vp_c>.
\end{align*}
In view of 
\begin{equation*}
    (\vp_c)_t= (\vp_t)_c+c(\vp_x)_c, \quad (\vp_x)_c=(\vp_c)_x,
\end{equation*}
we obtain
\begin{align*}
    &<\bu,\vp_t> - <\bu+7\bu^2,\vp_x>+<(1-\partial_x^2)^{-1}R(\bu),\vp_x> \\
      &= <U,(\vp_c)_t - c(\vp_c)_x> - <U+7U^2,(\vp_c)_x>+<(1-\partial_x^2)^{-1}R(U),(\vp_c)_x>.
\end{align*}
Observe that, since $U$ is independent of time, we have 
\begin{align*}
     <U,(\vp_c)_t > &= \iint_{\R_+ \times \R} U(x-ct)\partial_t\vp_c(t,x) dt dx \\
%                            &=  \iint_{\R_+ \times \R} U(s)\partial_t\vp_c(t,s+ct) dt ds \\
                            &=  \int_{\R_+}U(s)\left(\int_{\R} \partial_t\vp_c(t,s+ct) dt\right) ds \\
                            &=  \int_{\R_+}U(s)\left(\vp_c(T,s+cT)-\vp_c(0,s)\right) ds=0, 
\end{align*}
where  $T$ is large enough to be outside of the support of $\vp_c$. Therefore, 
\begin{align*}
   &<\bu,\vp_t> - <\bu+7\bu^2,\vp_x>+<(1-\partial_x^2)^{-1}R(\bu),\vp_x> \\
   &= -c<U,(\vp_c)_x> - <U+7U^2,(\vp_c)_x>+<(1-\partial_x^2)^{-1}R(U),(\vp_c)_x> \\
   &=\iint_{\R_+ \times \R}  -(cU+U+7U^2)(\vp_c)_x + (1-\partial_x^2)^{-1}R(U)(\vp_c)_x dt dx =0,
\end{align*}
since $U\in H^1$ satisfies \eqref{eq_MASETW} with $\psi(x)=\vp_c(t,x)$ which belongs to $\C_0^{\infty}(\R)$ for every given  $t\geq 0$. This shows that $\bu$ is a weak solution  of \eqref{eq_MASE} which proves the claim and concludes the proof. 
\end{proof}

\begin{rem}
\label{rem}
Theorem \ref{MT} states that all $x$-symmetric solutions of equation \eqref{eq_MASE} are traveling wave solutions. Conversely, we know from \cite{Gasull2014} and \cite{GeyMan15} that all smooth, peaked and cusped solitary traveling waves are symmetric with respect to their crests or troughs, and likewise for periodic traveling waves on each period as long as they conserve the energy of the associated planar system. Indeed, for smooth solutions this follows immediately from the fact that the planar system corresponding to the ordinary differential equation satisfied by the traveling wave solutions of \eqref{eq_MASE} has a first integral which is symmetric in its second variable, and smooth traveling wave solutions of \eqref{eq_MASE}  correspond to smooth homoclinic or periodic orbits of the planar differential system  which are symmetric with respect to the horizontal axis and defined on the entire real line. 
On the other hand, orbits corresponding to solutions of \eqref{eq_MASE}  whose first derivative is not continuous or blows up in a countable number of points are heteroclinic or homoclinic connections whose alpha or omega limits tend to or reach a line of singularities in the phase plane. The corresponding solutions are therefore defined only on a subset of $\R$ and may be continued to the whole line by joining these wave segments, cf.~\cite{Gasull2014,GeyMan15}. In this way, it is possible to obtain a plethora of weak solutions with complex shapes, as demonstrated by Lenells \cite{Lenells2005a,Lenells2005} for the Camassa-Holm and Degasperis-Procesi, which are in general clearly not symmetric. However, if one requires that this continuation of solutions conserves the energy of the planar system, i.e. if one imposes that the orbit remains in the same level set of the first integral, then the solution composed of wave segments in this way is necessarily symmetric. The same is true for the Camassa-Holm equation: the weak traveling wave solutions which are not symmetric are the ones composed of wave segments corresponding to orbits with different energy. 
\end{rem}

\subsection*{Acknowledgements}
The author is supported by the FWF project J 3452 ''Dynamical Systems Methods in Hydrodynamics`` of the Austrian Science Fund. 
Part of this work was carried out during a research visit at NTNU Trondheim for which the author gratefully acknowledges support.

\tbe

\bibliographystyle{siam}

\end{document}